\documentclass{article}

\usepackage{float}
\usepackage{amssymb}
\usepackage{latexsym}
\usepackage{extarrows}
\usepackage{float}
\usepackage{graphicx}
\usepackage{indentfirst}

\newtheorem{theo}{Theorem}[section]
\newtheorem{remark}[theo]{Remark}

\newtheorem{lemma}[theo]{Lemma}

\newtheorem{coro}[theo]{Corollary}

\newtheorem{prop}[theo]{Proposition}

\newtheorem{defi}[theo]{Definition}

\newcommand{\qed}{\hspace*{\fill} \rule{7pt}{7pt}}

\def\endproofbox{\hskip 1.3em\hfill\rule{6pt}{6pt}}

\newenvironment{proof}%
{%
\noindent{\it Proof.}
}%
{%
 \quad\hfill\endproofbox\vspace*{2ex}
}

\def\qed{\hskip 1.3em\hfill\rule{6pt}{6pt}}

\begin{document}
\title{The Ramsey Number for a Forest versus Disjoint Union of Complete Graphs \thanks{Supported  by National Natural Science Foundation of China  (No. 11931002).}}
\author{Sinan Hu \thanks{School of Mathematics, Hunan University, Changsha 410082, P.R. China. Email: husinan@hnu.edu.cn.}
\and   Yuejian Peng \thanks{Corresponding author. School of Mathematics, Hunan University, Changsha, 410082, P.R. China. Email: ypeng1@hnu.edu.cn.}}
\date{}
\maketitle
\begin{abstract}
Given two graphs $G$ and $H$, the Ramsey number $R(G,H)$ is the minimum integer $N$ such that  any coloring of the edges of $K_N$ in red or blue yields a red $G$ or a blue $H$.
Let $v(G)$ be the number of vertices of $G$ and $\chi(G)$ be the chromatic number of $G$.
Let $s(G)$ denote the chromatic surplus of $G$, the cardinality of a minimum  color class taken over all proper colorings of $G$ with $\chi(G)$ colors.
Burr \cite{B} showed that for a connected graph $G$ and a graph $H$ with $v(G)\geq s(H)$,
$R(G,H) \geq (v(G)-1)(\chi(H)-1)+s(H)$.
A connected graph $G$ is called $H$-good if $R(G,H)=(v(G)-1)(\chi(H)-1)+s(H)$.
Chv\'atal  \cite{C} showed that any tree is $K_m$-good for $m\geq2$, where $K_m$ denotes a complete graph with $m$ vertices. 
By applying this result, Stahl \cite{St}  determined the Ramsey number of a forest versus $K_m$. 
Concerning whether a tree is $H$-good for $H$ being disjoint union of complete graphs, Chv\'atal and Harary \cite{CH} showed that any tree is $2K_2$-good, where $tH$ denotes the union of $t$ disjoint copies of graph $H$.  
Sudarsana, Adiwijaya and Musdalifah \cite{SAM} proved that the $n$-vertex path $P_n$ is $2K_m$-good for $n\geq3$ and $m\geq2$, and conjectured that any tree $T_n$ with $n$ vertices is $2K_m$-good.  
Recently, Pokrovskiy and Sudakov \cite{PS} proved that $P_n$ is $H$-good for a graph with $n\geq4v(H)$.  
Balla, Pokrovskiy and  Sudakov \cite{BPS} showed that  for all $\Delta$ and $k$, there exists a constant $C_{\Delta,k}$ such that for any tree $T$ with maximum degree at most $\Delta$ and any $H$ with $\chi(H)=k$ satisfying $v(T)\geq C_{\Delta,k}v(H)\log^4 v(H)$, $T$ is $H$-good. 

In this paper, we explore the Ramsey number of  forest versus disjoint union of complete graphs. 
We first confirm the conjecture by Sudarsana, Adiwijaya and Musdalifah \cite{SAM} that any tree is $2K_m$-good for $n\geq3$ and $m\geq2$. 
A  key proposition in our proof is that a tree with $n$ vertices can be obtained from any tree with $n$ vertices by performing a series of “Stretching”  and “Expanding” operations. 
On this foundation, we show that these two operations preserve the “$2K_m$-goodness” property, and confirm that any tree is $2K_m$-good.  
We also prove a  conclusion which yields that $T_n$ is $K_m\cup K_l$-good, where $K_m\cup K_l$ is the disjoint union of $K_m$ and $K_l$, $m>l\geq2$. 
Furthermore, 
we extend the Ramsey goodness of  connected graphs to disconnected graphs and  study the relation between the Ramsey number of the components of  a disconnected graph $\mathrm{F}$ versus a graph $H$. 
We show that if each component of a graph $F$ is $H$-good, then $F$ is $H$-good. 
Our result implies the exact value of $R(F,K_m\cup K_l)$, where $F$ is a forest and $m,l\geq2$.  
It's interesting to explore what kind of graph $F$ can satisfy that $F$ is $H$-good and $G$-good implies that $F$ is $H\cup G$-good when the number of vertices in $F$ is sufficiently large. 
\end{abstract}


\section{Introduction}
Graphs considered in this paper are finite, undirected and simple.
For a graph $G$, let $V(G)$ be the vertex set of $G$ and $E(G)$ be the edge set of $G$.
Let $v(G)$ denote the order of $G$, i.e., $|V(G)|=v(G)$.
For $u\in V(G)$,  $N(u)=\{v: uv\in E(G)\}$, $N[u]=\{u\}\cup N(u)$ and $d(u)=|N(u)|$. 
Denote the chromatic number of $G$ by $\chi(G)$, and let $s(G)$ denote the chromatic surplus of $G$, which is the number of vertices in a minimum  color class over all proper $\chi(G)$-colorings of $V(G)$. 
Let $U$ be a subset of $V(G)$.
Denote $G-U$ as the graph obtained from $G$ by deleting $U$ and all edges incident to $U$.
Denote a complete graph on $n$ vertices by $K_n$, and a path on $n$ vertices by $P_n$. The union of $k$ disjoint copies of a graph $F$ is denoted by $kF$ and the disjoint union of $G$ and $H$ is denoted by $G\cup H$. The join graph of $G$ and $H$, denoted by $G\vee H$, is  the graph obtained from the vertex disjoint union $G\cup H$ by joining each vertex of $G$ to each vertex of $H$.  

Let $G$ and $H$ be graphs without isolated vertices.
The Ramsey number $R(G,H)$ is the minimum integer $N$ such that  any coloring of the edges of $K_N$ in red or blue yields a red $G$ or a blue $H$.
Determining $R(G,H)$ in general is a very challenging problem, there are several excellent surveys on Ramsey numbers. 
In this paper,  we consider a problem related to Ramsey goodness. 
In  \cite{B}, Burr gave the following lower bound.
\begin{theo}\label{Burr} {\rm(Burr \cite{B})}
For a connected graph $G$ and a graph $H$ with $v(G)\geq s(H)$,
$$R(G,H) \geq (v(G)-1)(\chi(H)-1)+s(H).$$
\end{theo}

Burr defined $G$ to be $H$-good if the equality $$R(G,H) = (v(G)-1)(\chi(H)-1)+s(H)$$ holds under the conditions of Theorem \ref{Burr}. 

In the 1970s, before the definition of Ramsey-goodness  was given, a well-known result of Chv\'atal \cite{C} showed that any tree is $K_m$-good for every $m\geq2$ and an earlier result of Chv\'atal and Harary \cite{CH} showed that any tree is $2K_2$-good. 
By applying Chv\'atal's theorem, Stahl \cite{St}  determined the Ramsey number of a forest versus $K_m$. 
In 1983,  Burr and Erd\H{o}s \cite{BE} proved that for any fixed $k$ and $m$, there exists $n_0$ such that the family of connected graphs with bandwidth at most $k$ and at least $n_0$ vertices is $K_m$-good, where the bandwidth of a graph $G$ is the smallest number $k$ such that there is an ordering $v_1,\dots,v_n$ of $V(G)$ such that each edge $v_iv_j$ satisfies $|v_i-v_j|\leq k$. 
This result was recently extended by Allen, Brightwell and Skokan \cite{ABS} that for each fixed graph $H$ and $k$,  there exists $n_0$ such that the family of connected graphs with bandwidth at most $k$ and at least $n_0$ vertices is $H$-good. 
These two results implies a path (with bandwidth  $1$) or a cycle (with bandwidth  $2$) with sufficiently  large number of vertices has good goodness properties. 
Results without the assumption of  sufficiently  large number of vertices are also interesting. 
In \cite{SAM}, Sudarsana, Adiwijaya and Musdalifah  showed that $P_n$ is $2K_m$-good for $n\geq3$ and $m\geq2$, and conjectured that any tree $T_n$ with $n$ vertices is $2K_m$-good. 
Recently, Pokrovskiy and Sudakov \cite{PS} proved that for a fixed graph $H$, the path on $n$ vertices  with $n\geq4v(H)$ is $H$-good. 
Balla, Pokrovskiy and  Sudakov \cite{BPS} showed that for all $\Delta$ and $k$, there exists a constant $C_{\Delta,k}$ such that for any tree $T$ with maximum degree at most $\Delta$ and any $H$ with $\chi(H)=k$ satisfying $v(T)\geq C_{\Delta,k}v(H)\log^4 v(H)$, $T$ is $H$-good. 
In \cite{LLD}, Lin,   Li and  Dong proved that if $T_n$ is $G$-good and $s(G)=1$, then $T_n$ is $K_1\vee G$-good. 
For other results concerning Ramsey-goodness graphs, we refer the reader to the survey papers by Conlon, Fox and Sudakov \cite{CFS}, and Radziszowski \cite{R}. 
In this paper, we consider the Ramsey number of a forest versus disjoint union of complete graphs. 
In Section $2$, we show  the following result  which is  a continuation of Chv\'atal's  classical result. 
It also  confirms the mentioned  conjecture  by Sudarsana, Adiwijaya and Musdalifah \cite{SAM}. 

\begin{theo}\label{Tn}
Let $n\geq 3$ and $m\geq2$ be integers. Let $T_n$ be a tree with order $n$, then $$R(T_n,2K_m)=(n-1)(m-1)+2.$$
\end{theo}


The  proof of Theorem \ref{Tn} will be given in Section $2$. 
In Section $3$, we show a  result which yields that $T_n$ is $K_m\cup K_l$-good, where $n\geq 3$ and $m>l\geq2$ are integers. 
In Section $4$, we introduce a construction which yields a general lower bound of $R(G,H)$ for arbitrary graphs $G$ and $H$.  
On  this foundation, we extend the Ramsey goodness of  connected graphs to disconnected graphs and  explore the relation between the Ramsey number of a disconnected graph $\mathrm{F}$ versus a graph $H$ and the Ramsey number of its components versus $H$. 
We extend an upper bound given by Gould and Jacobson \cite{GJ},  and show that if each component of a graph $F$ is $H$-good, then $F$ is $H$-good. 
Furthermore, we  will apply the Ramsey goodness results of trees to obtain the Ramsey number of a forest versus disjoint union of complete graphs.  Next, we outline the  idea of the proof of Theorem \ref{Tn}.

\subsection{Two Operations on a Tree and Outline of the Proof of Theorem \ref{Tn}}
A  key observation in the proof of Theorem \ref{Tn} is  that any  tree with $n$ vertices can be obtained from $P_n$ by performing a series of two operations. 
Furthermore, we show that these two operations preserve the “$2K_m$-goodness” property. 
Let us describe these two operations  precisely below.

\

\noindent{\bf Stretching a tree $T$ at a leaf $a$ :} {\em Let $T$ be a tree with $n\geq 3$ vertices and $a$ be a leaf in $T$.
Let $T'$ be obtained from $T$ by deleting a leaf $b$ (other than $a$) of $T$ and adding a new vertex connecting to $a$. We say that $T'$ is obtained by Stretching $T$ at  leaf $a$.}

\

\noindent{\bf Expanding a tree $T$ at a vertex $u$ :} {\em Let $2\leq d\leq n-2$ and $T$ be a tree with $n\geq4$ vertices.  Let $u\in V(T)$ with $N(u)=\{z_0,z_1,\dots,z_{d-1}\}$ satisfying that $z_i$ is a leaf of $T$ for each $i\in[d-1]$ and $z_0$ is not a leaf.
Let $T'$ be obtained from $T$ by deleting a leaf $b\in T- N[u]$ and adding a new vertex connecting to $u$. We say that $T'$ is obtained by Expanding $T$ at vertex $u$.}

\begin{defi}\label{defi}
We say that a property $\mathcal{P}$ is Stretching-preserving (or Expanding-preserving) if a tree $T$ satisfying $\mathcal{P}$ implies that $T'$ also satisfying  $\mathcal{P}$, where $T'$ is obtained by Stretching (or Expanding) $T$ at a leaf (or a vertex).
\end{defi}

 In Section $2.2$, we will prove the following key observation.
\begin{prop}\label{remark}
Given a tree $T$ on $n$ vertices, we can obtain any tree on $n$ vertices from $T$ by applying Stretching and Expanding multiple times.
\end{prop}

It's easy to see that  Proposition \ref{remark} will imply the following corollary.

\begin{coro}\label{111}
If a tree $T$ satisfies property $\mathcal{P}$, and $\mathcal{P}$ is Stretching-preserving  and Expanding-preserving, then every tree satisfies $\mathcal{P}$.
\end{coro}

In Section $2$, we show that the property  “$2K_m$-goodness” is Stretching-preserving  and Expanding-preserving, and prove Theorem \ref{Tn}.

\section{Tree is $2K_m$-good}
\subsection{Proof of Theorem \ref{Tn}}
In this subsection we prove Theorem \ref{Tn}. 
We first prove the  crucial lemmas that the property “$2K_m$-goodness” is Stretching-preserving and Expanding-preserving  under some conditions. 
\begin{lemma}\label{1}
Let $n\geq 3$ and $m\geq2$ be integers.
Assume that $R(T,2K_{m-1})\leq (n-1)(m-2)+2$ holds for any tree $T$ with order $n$.
Let $T_n^*$ be a tree with $n$ vertices and $T_n^{**}$ be obtained by Stretching $T_n^{*}$ at a leaf $a$  (see Figures \ref{T3} and \ref{T4}).
If $R(T_n^*,2K_m)\leq (n-1)(m-1)+2$, then $R(T_n^{**},2K_m)\leq (n-1)(m-1)+2$.
 \end{lemma}
\begin{proof}
If $n=3$, a tree with $3$ vertices must be $P_3$ and the result holds.
If $T_n^*=P_n$ is a path, then $T_n^{**}=P_n$ and the result holds.
So we may assume that $n\geq4$ and $T_n^*$ is not a path.

Let $u$ be the vertex adjacent to $a$ in $T_n^*$.
Since $T_n^*$ is not a path and $n\geq4$, $T_n^*$ has at least two leaves $b$ and $c$ in $V(T_n^*)\setminus\{a,u\}$.
Let $vc$ be an edge (as in Figure \ref{T3}).
We remark that it's possible that $v=u$.
\begin{figure}[H]
\centering
\begin{minipage}{5cm}
\includegraphics[width=0.9\textwidth, height=0.7\textwidth]{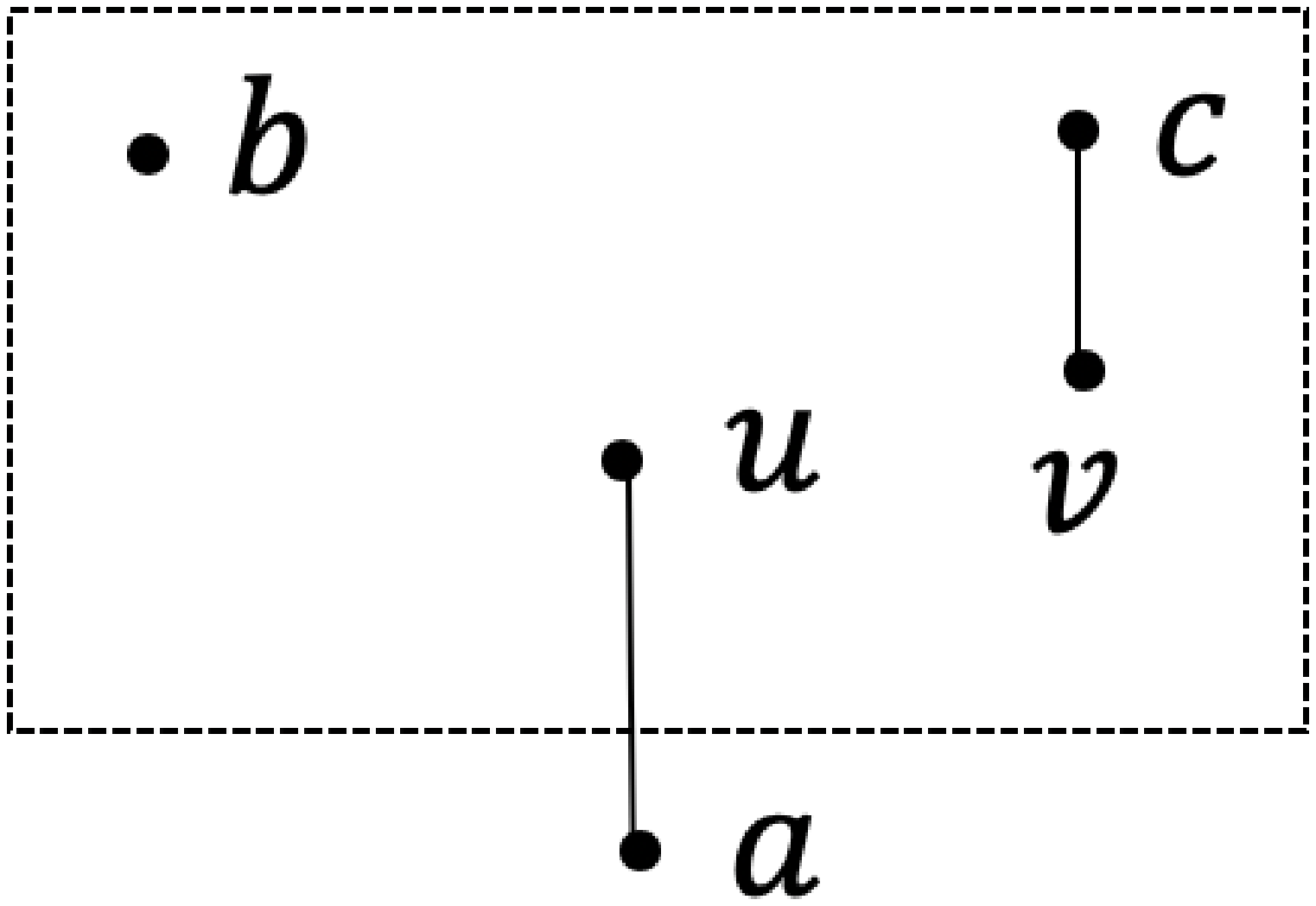}
\caption{$T_n^*$}
   \label{T3}
\end{minipage}
\begin{minipage}{5cm}
\includegraphics[width=0.87\textwidth, height=0.7\textwidth]{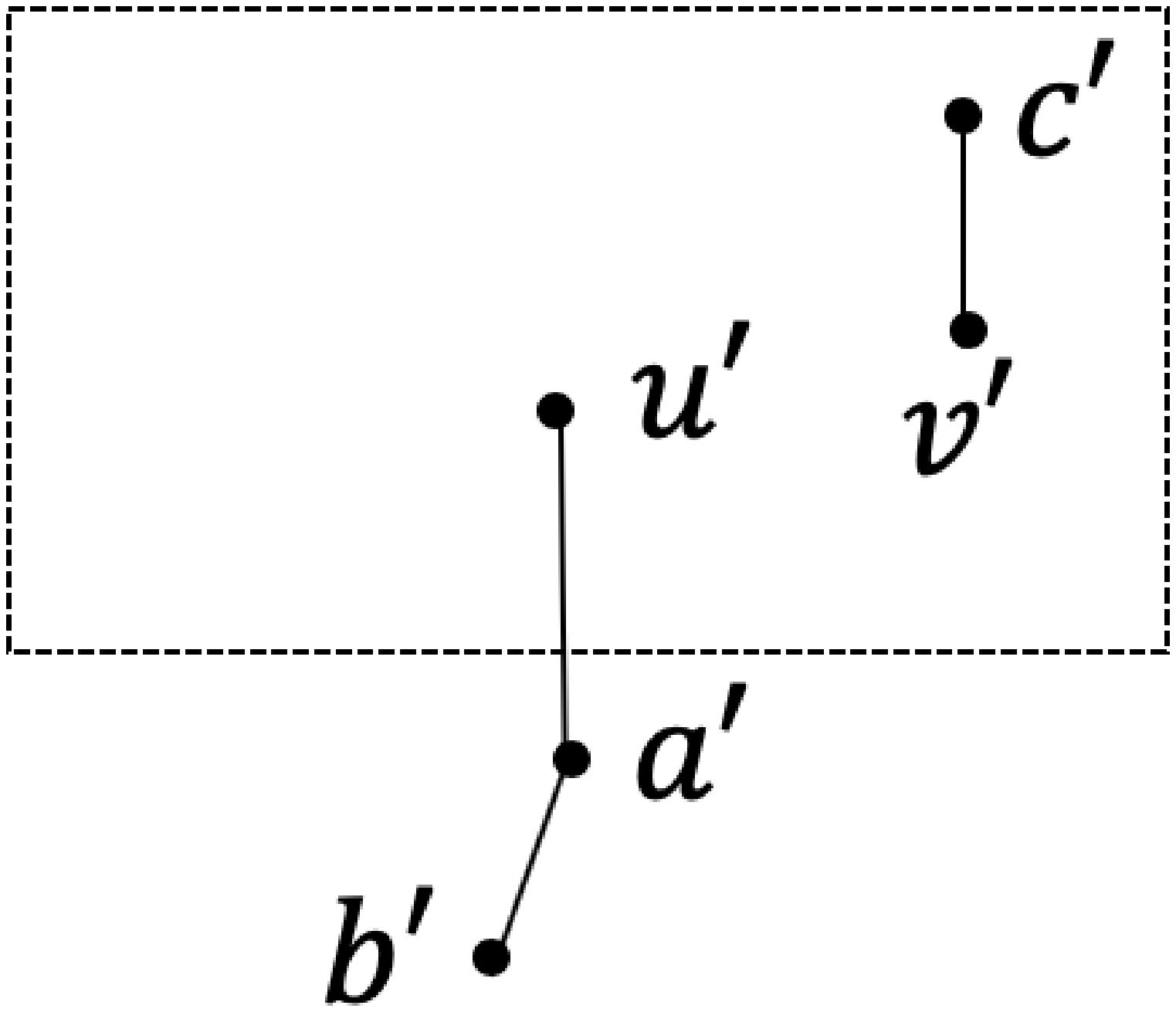}
\caption{$T_n^{**}$}
   \label{T4}
\end{minipage}
\end{figure}
\noindent
Let $N=(n-1)(m-1)+2$ and the edges of $K_N$ be colored by red or blue.
We will show that $K_N$ contains a red $T_n^{**}$ or a blue $2K_m$.
Since $R(T_n^*,2K_m)\leq (n-1)(m-1)+2$, then $K_N$ contains a red $T_n^{*}$ or a blue $2K_m$.
We just need to consider that $K_N$ contains a red $T_n^{*}-\{b\}\subseteq T_n^{*}$, whose order is $n-1$. 
Since $N-(n-1)=(n-1)(m-2)+2$ and $R(T,2K_{m-1})\leq (n-1)(m-2)+2$ holds for any tree $T$ with order $n$, then $K_N-(T_n^{*}-\{b\})$ contains a red $T_n^{**}$ or a blue $2K_{m-1}$.
We just need to consider that $K_N-(T_n^{*}-\{b\})$ contains a blue $2K_{m-1}$, denoted by $A$ and $B$ (see Figure \ref{T5}).
We note that the edges between $a$ and $A$, and $a$ and $B$ are all blue.
Otherwise there will be a red copy of $T_n^{**}$ and we are done.
Let $F=K_N-(T_n^{*}-\{b\})-A-B+\{c\}$.
Clearly $|V(F)|=(n-3)(m-2)+1$.
By Chv\'atal's theorem, $F$ contains a red $T_{n-2}=T_n^{**}-\{b',c'\}$ or a blue $K_{m-1}$.

Case $1$. $F$ contains a red $T_n^{**}-\{b',c'\}$ as in Figure \ref{T5}.
Note that there exist $x_A'\in A$ and $x_B'\in B$ such that $x_A'v'$ and $x_B'a'$ are red.
Otherwise, $\{v'\}\cup A$ and $\{a\}\cup B$ or $\{a'\}\cup B$ and $\{a\}\cup A$ form a blue copy of $2K_m$, we are done.
Hence $T_n^{**}-\{b',c'\}+\{x_A',x_B'\}$ contains a red copy of $T_n^{**}$ (see Figure \ref{T5}), and we are done.
\begin{figure}[H]
\centering
    \includegraphics[height=5.5cm]{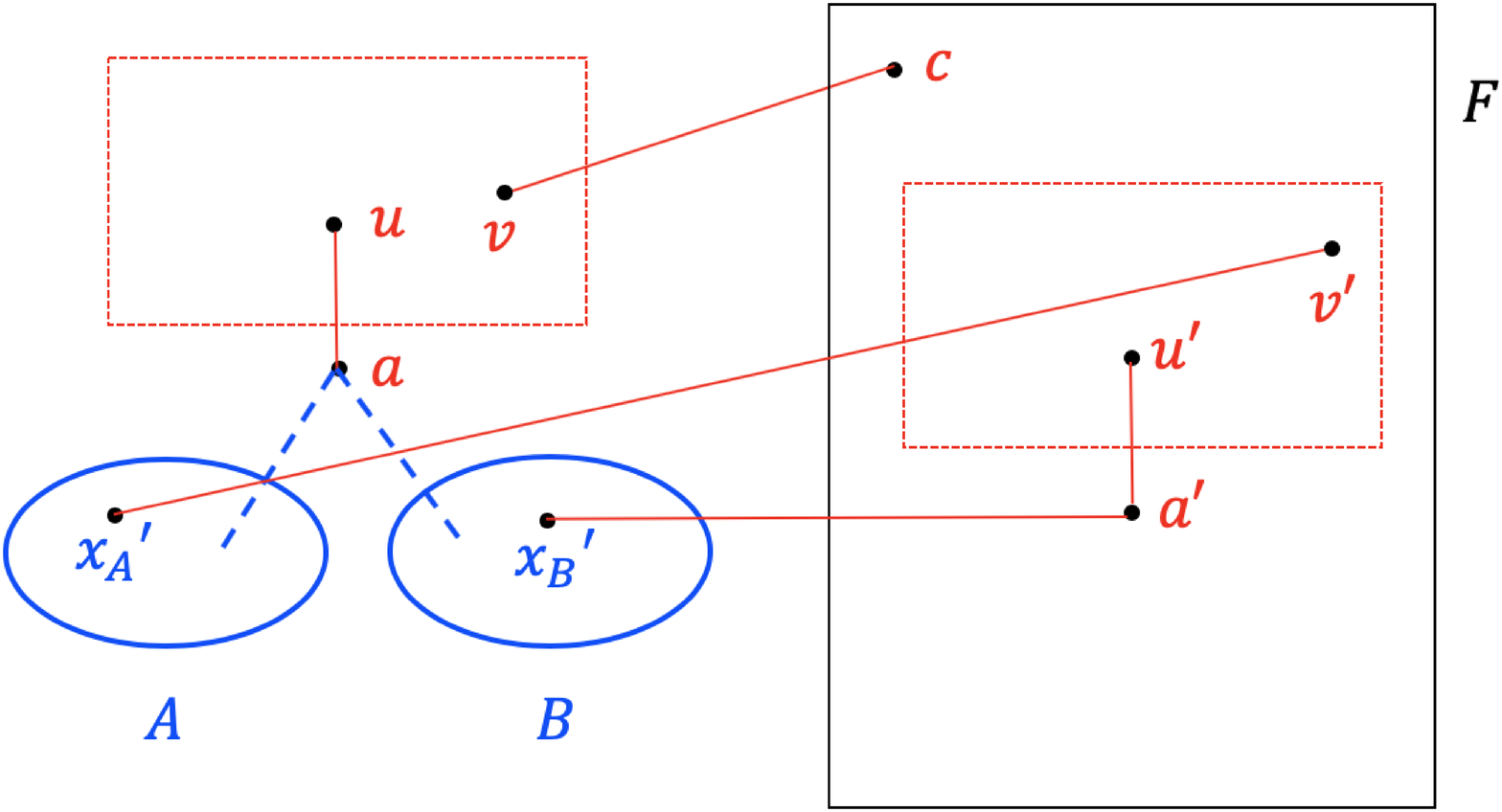}
    \caption{$F$ contains a red $T_n^{**}-\{b',c'\}$}
        \label{T5}
    \end{figure}

Case $2$. $F$ contains a blue $K_{m-1}$, denoted by $C$ (see Figure \ref{T6}).
To avoid a blue copy of $2K_m$, there exist $x_A\in A$ and $x_B\in B$ such that $ux_A$ and $vx_B$ are red.
Otherwise, $\{u\}\cup A$ and $\{a\}\cup B$ or $\{v\}\cup B$ and $\{a\}\cup A$ form a blue copy of $2K_m$, we are done.
Furthermore, there exists $x_C\in C$ such that $x_Ax_C$ is red.
Otherwise, $\{x_A\}\cup C$ and $\{a\}\cup B$ form a blue copy of $2K_m$, we are done.
Hence $T_n^{*}-\{b\}-\{a,c\}+\{x_A,x_B,x_C\}\subseteq K_N$ contains a red copy of $T_n^{**}$ (see Figure \ref{T6}). and we are done.
\begin{figure}[H]
\centering
    \includegraphics[height=5.5cm]{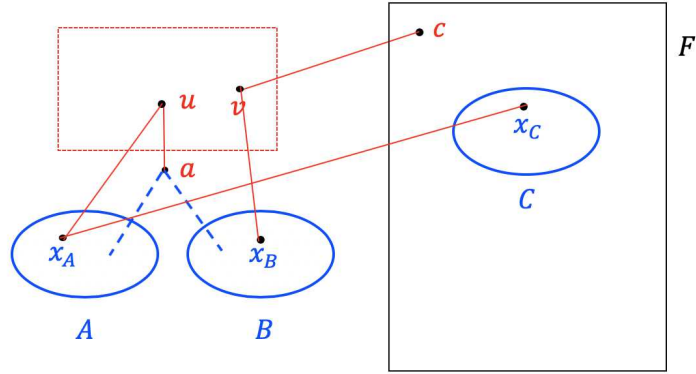}
    \caption{$F$ contains a blue $K_{m-1}$}
        \label{T6}
    \end{figure}
\end{proof}

\begin{lemma}\label{2}
Let $n\geq 3$ and $m\geq2$ be integers.
Assume that $R(T,2K_{m-1})\leq (n-1)(m-2)+2$ holds for any tree $T$ with order $n$.
Let $T_n^*$ be a tree with $n$ vertices.
Let $T_n^{**}$ be obtained by Expanding $T_n^{*}$ at a vertex $u$ (see Figures \ref{T1} and \ref{T2}).
If $R(T_n^*,2K_m)\leq (n-1)(m-1)+2$, then $R(T_n^{**},2K_m)\leq (n-1)(m-1)+2$.
\begin{figure}[H]
\centering
\begin{minipage}{5cm}
\includegraphics[width=0.85\textwidth, height=0.8\textwidth]{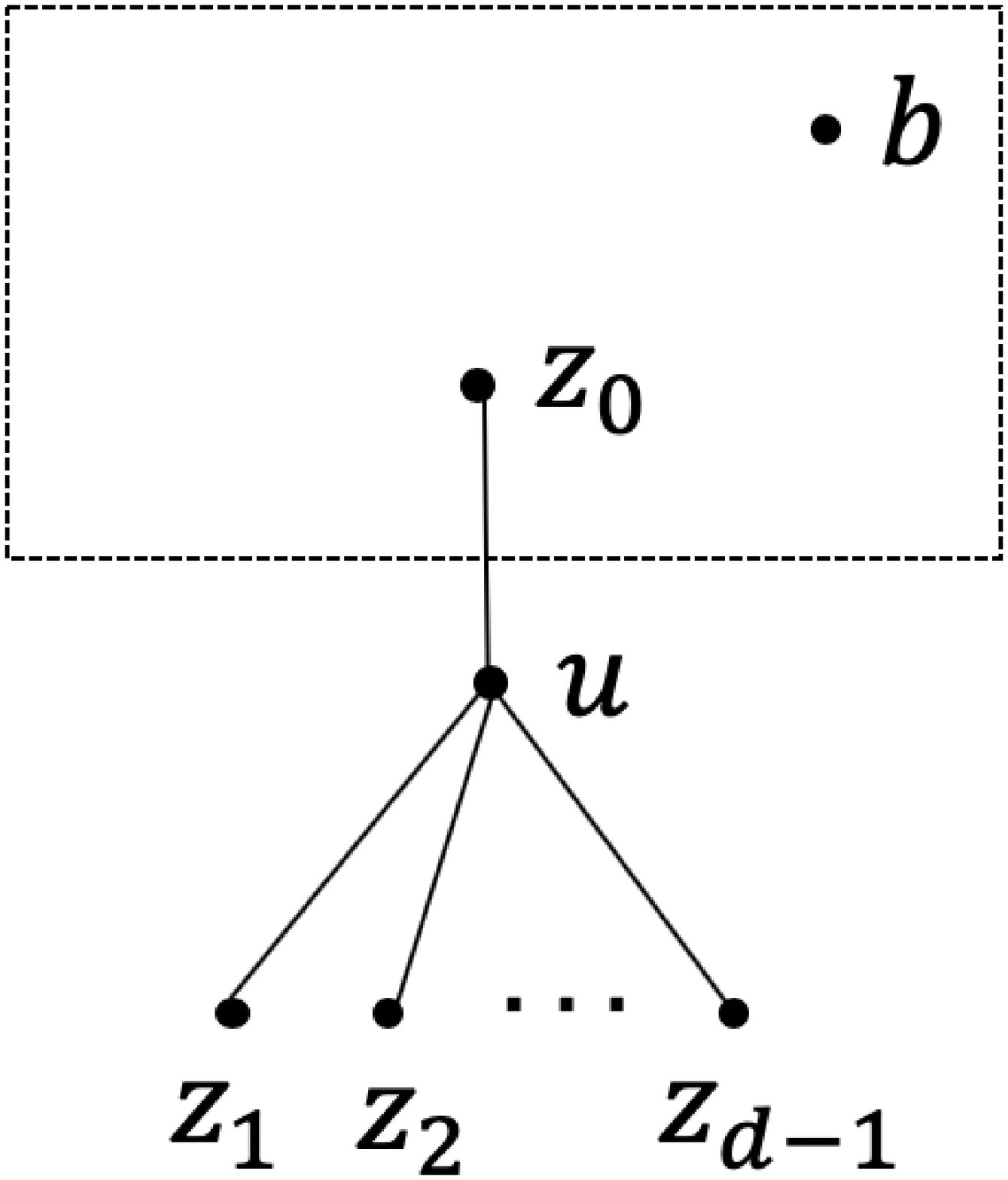}
\caption{$T_n^*$}
   \label{T1}
\end{minipage}
\begin{minipage}{5cm}
\includegraphics[width=0.87\textwidth, height=0.8\textwidth]{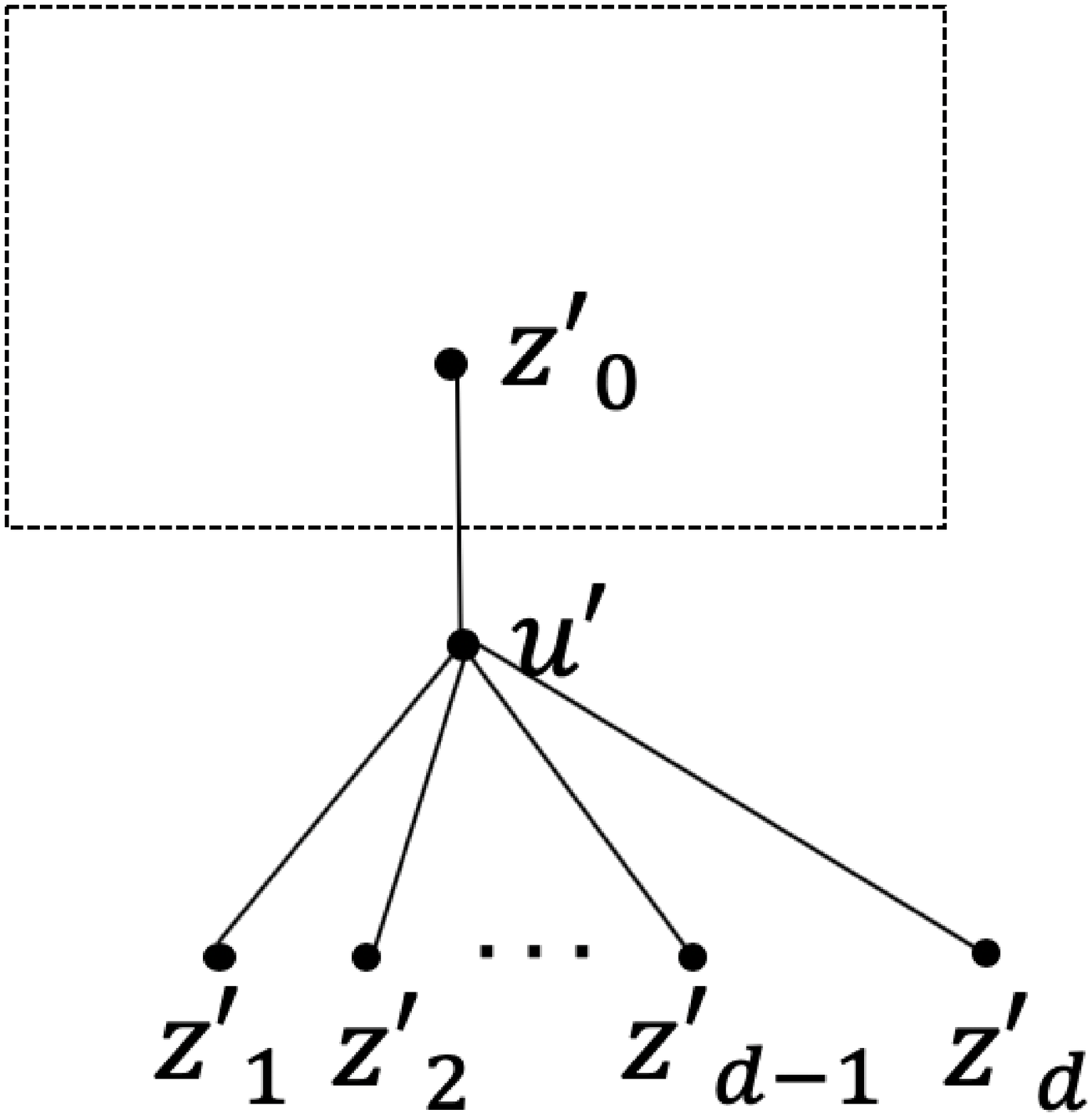}
\caption{$T_n^{**}$}
   \label{T2}
\end{minipage}
\end{figure}
\end{lemma}
\begin{proof}
We just need to consider that $n\geq4$.
Since $T_n^*- N_1(u)+\{z_0\}$ is also a tree, then there exists at least a leaf $b$ other than $z_0$ in $T_n^*- N_1(u)$.
Note that $b$ is also a leaf of $T_n^*$.
Recall that $d\leq n-2$ in view of the definition of Expanding.
Let $N=(n-1)(m-1)+2$ and the edges of $K_N$ be colored by red or blue.
We will show that $K_N$ contains a red $T_n^{**}$ or a blue $2K_m$.
Since $R(T_n^*,2K_m)\leq (n-1)(m-1)+2$,
then $K_N$ contains a red $T_n^{*}$ or a blue $2K_m$.
We just need to consider that $K_N$ contains a red $T_n^{*}-\{b\}\subseteq T_n^{*}$,
whose order is $n-1$.
Since $N-(n-1)=(n-1)(m-2)+2$ and $R(T,2K_{m-1})\leq (n-1)(m-2)+2$ holds for any tree $T$ with order $n$,
then $K_N-(T_n^{*}-\{b\})$ contains a red $T_n^{**}$ or a blue $2K_{m-1}$.
We just need to consider that $K_N-(T_n^{*}-\{b\})$ contains a blue $2K_{m-1}$, denoted by $A$ and $B$ (see Figure \ref{T7}).
We note that the edges between $u$ and $A$, and $u$ and $B$ are all blue.
Otherwise there will be a red copy of $T_n^{**}$ and we are done.
\begin{figure}[H]
\centering
    \includegraphics[height=4cm]{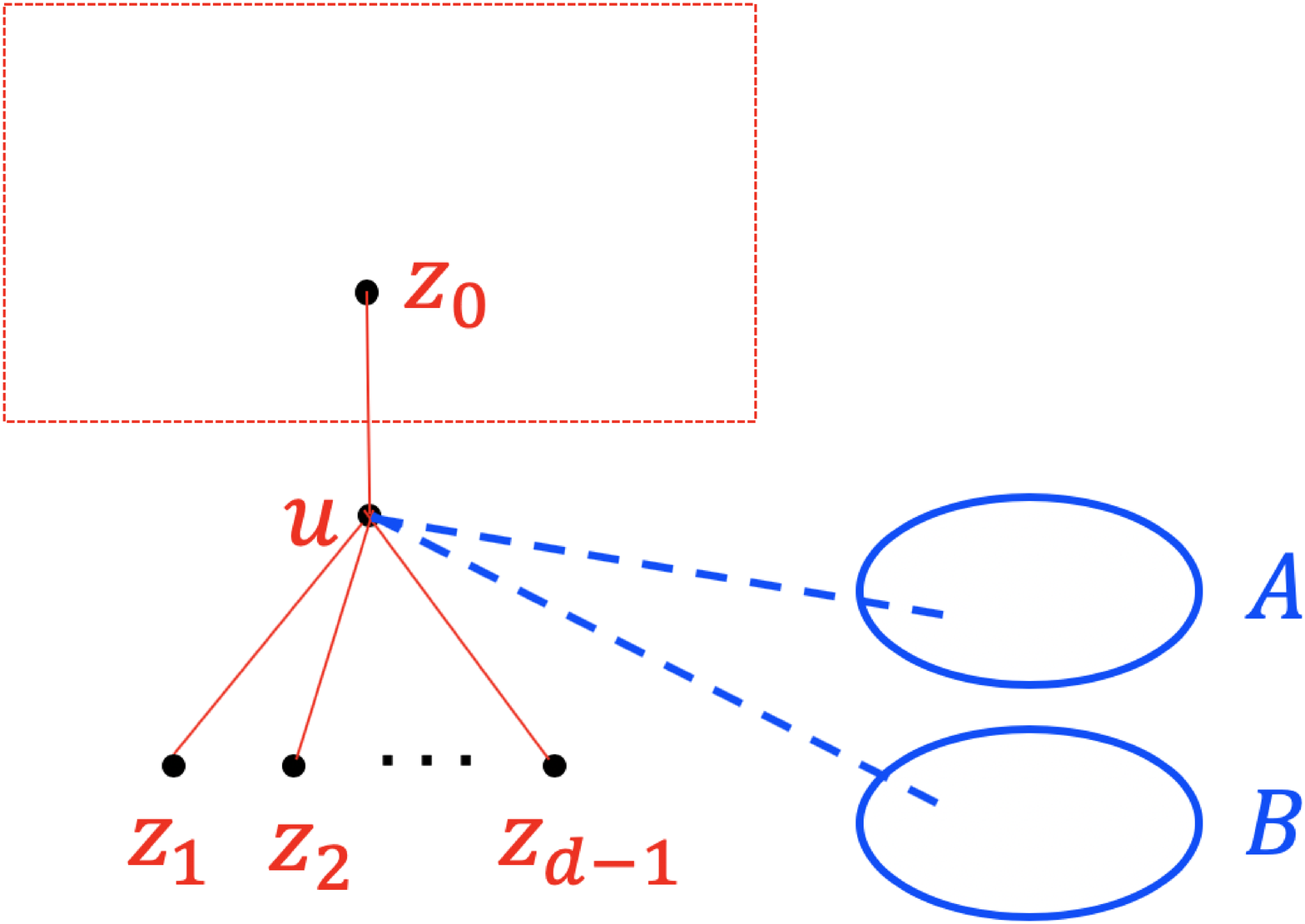}
    \caption{$K_N-(T_n^{*}-\{b\})$ contains a blue $2K_{m-1}$}
        \label{T7}
    \end{figure}

Let $H_1=K_N-(T_n^{*}-\{b\})-A-B+\{z_1\}$.
Then $|V(H_1)|=(n-3)(m-2)+1$.

If $d\geq3$, we perform the following procedure.
By Chv\'atal's theorem, $H_1$ contains a red $T_{n-2}=T_n^{**}-\{z_1',z_2'\}$ or a blue $K_{m-1}$.
Let us consider that $H_1$ contains a red $T_n^{**}-\{z_1',z_2'\}$ first, as in Figure \ref{T8}.
Note that there exist $x_A'\in A$ and $x_B'\in B$ such that $x_A'u'$ and $x_B'u'$ are red.
Otherwise, $\{u'\}\cup A$ and $\{u\}\cup B$ or $\{u'\}\cup B$ and $\{u\}\cup A$ form a blue copy of $2K_m$, we are done.
Hence $T_n^{**}-\{z_1',z_2'\}+\{x_A',x_B'\}$ contains a red copy of $T_n^{**}$ (see Figure \ref{T8}), and we are done.
\begin{figure}[H]
\centering
    \includegraphics[height=5cm]{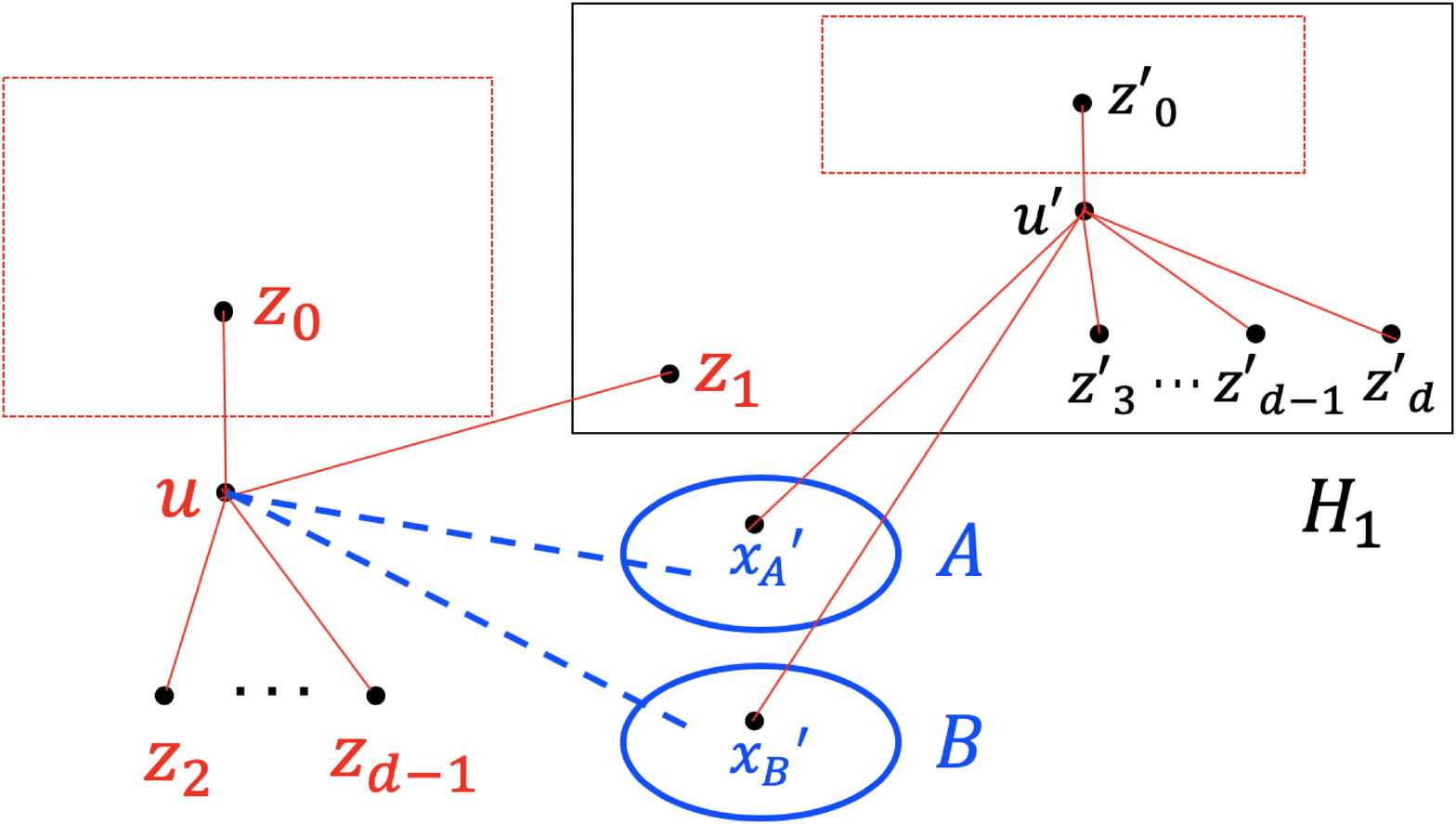}
    \caption{$H_1$ contains a red $T_n^{**}-\{z_1',z_2'\}$}
        \label{T8}
    \end{figure}
\noindent Now we consider that $H_1$ contains a blue $K_{m-1}$, denoted by $C_1$ (see Figure \ref{T9}).
Let $H_{2}=H_1-C_1+\{z_2\}$ (see Figure \ref{T9}).
Then $|V(H_2)|=(n-4)(m-2)+1$.
By Chv\'atal's theorem, $H_2$ contains a red $T_{n-3}=T_n^{**}-\{z_1',z_2',z_3'\}$ or a blue $K_{m-1}$.
Let us consider that $H_2$ contains a red $T_n^{**}-\{z_1',z_2',z_3'\}$ first, as in  Figure \ref{T9}.
Note that there exist $x_A'\in A$, $x_B'\in B$ and $x_{C_1}'\in C_1$ such that $x_A'u'$, $x_B'u'$ and $x_{C_1}'u'$ are red.
Otherwise, $\{u'\}\cup A$ and $\{u\}\cup B$ or $\{u'\}\cup B$ and $\{u\}\cup A$ or $\{u'\}\cup C_1$ and $\{u\}\cup A$ form a blue copy of $2K_m$, and we are done.
Hence $T_n^{**}-\{z_1',z_2',z_3'\}+\{x_A',x_B',x_{C_1}'\}$ contains a red copy of $T_n^{**}$ (see Figure \ref{T9}), and we are done.
\begin{figure}[H]
\centering
    \includegraphics[height=5cm]{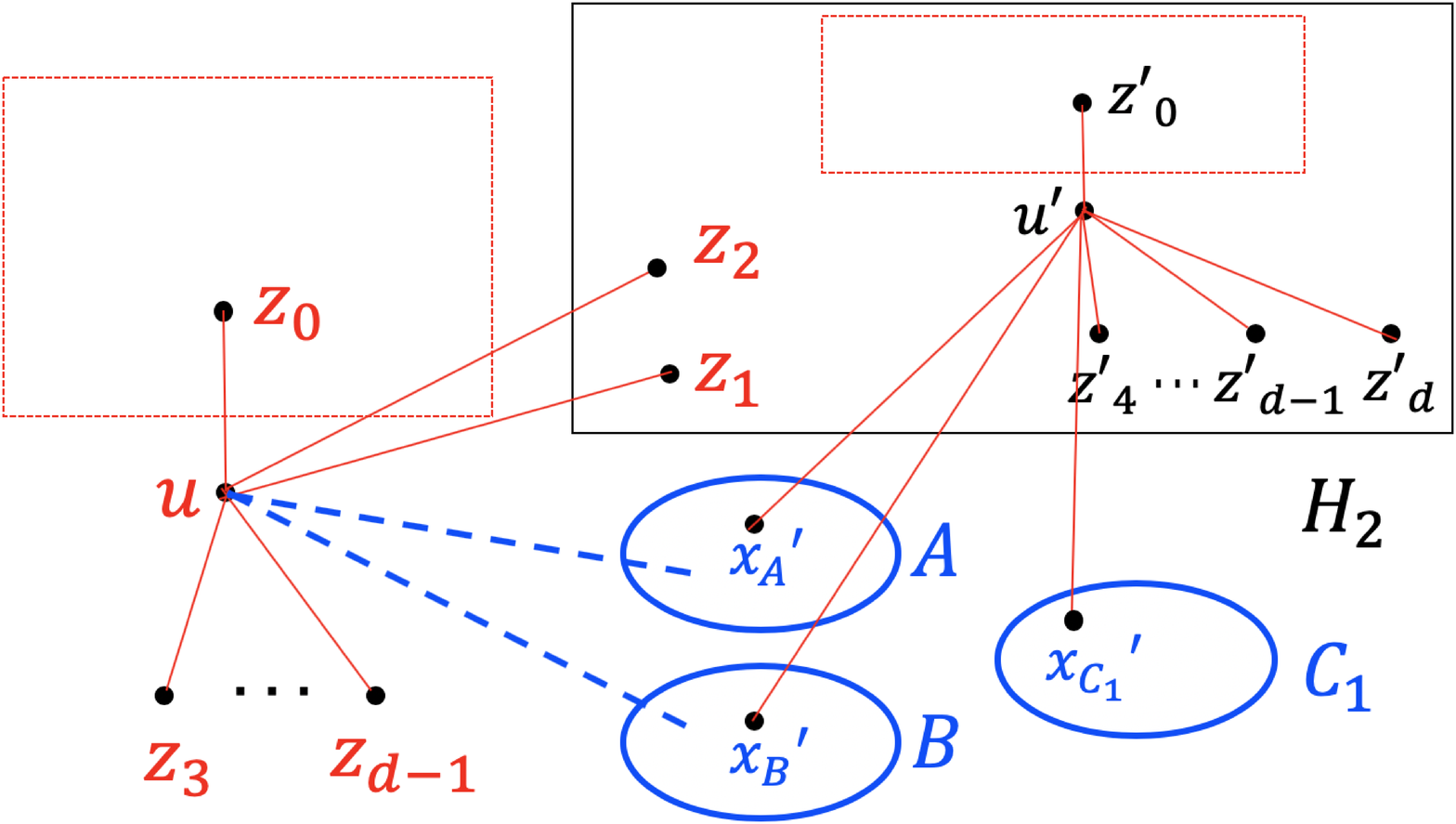}
    \caption{$H_2$ contains a red $T_n^{**}-\{z_1',z_2',z_3'\}$}
        \label{T9}
    \end{figure}
\noindent  We now consider that $H_2$ contains a blue $K_{m-1}$, denoted by $C_2$.
We take a similar procedure as before and continue the process at most $d-2$ times.
Either we are done or we obtain $d-2$ blue copies of $K_{m-1}$, denoted by $C_1,C_2,\dots,C_{d-2}$ (see Figure \ref{T10}) and $H_1,H_2,\dots,H_{d-2}$ with $H_{i+1}=H_i-C_i+\{z_{i+1}\}$ for $i=1,2,\dots,d-3$.
Let $H_{d-1}=H_{d-2}-C_{d-2}+\{z_{d-1}\}$.
If $d=2$, then $H_1=K_N-(T_n^{*}-\{b\})-A-B+\{z_1\}$, as in the beginning of the last paragraph.
Then $|V(H_{d-1})|=[n-2-(d-1)](m-2)+1=(n-d-1)(m-2)+1$.
Recall that $d\leq n-2$, then $|V(H_{d-1})|\geq m-1$.
By Chv\'atal's theorem, $H_{d-1}$ contains a red $T_{n-d}=T_n^{**}-\{z_1',z_2',\dots,z_d'\}$ or a blue $K_{m-1}$ (see Figure \ref{T10}).

Let us consider that $H_{d-1}$ contains a red $T_n^{**}-\{z_1',z_2',\dots,z_d'\}$ first.
Note that there exist $x_A'\in A,x_B'\in B,x_{C_1}'\in C_1,\dots,x_{C_{d-3}}'\in C_{d-3}$ and $x_{C_{d-2}}'\in C_{d-2}$, such that $x_A'u',x_B'u',x_{C_1}'u',\dots,x_{C_{d-3}}'u'$ and $x_{C_{d-2}}'u'$ are red.
Otherwise, 
$\{u'\}\cup A$ and $\{u\}\cup B$ or $\{u'\}\cup B$ and $\{u\}\cup A$ or $\{u'\}\cup C_i$ and $\{u\}\cup A$ form a blue copy of $2K_m$, where $i\in[d-2]$, and we are done.
Hence $T_n^{**}-\{z_1',z_2',\dots,z_d'\}+\{x_A',x_B',x_{C_1}',\dots,x_{C_{d-3}}',x_{C_{d-2}}'\}$ contains a red copy of $T_n^{**}$ (see Figure \ref{T10}), and we are done.
\begin{figure}[H]
\centering
    \includegraphics[height=5cm]{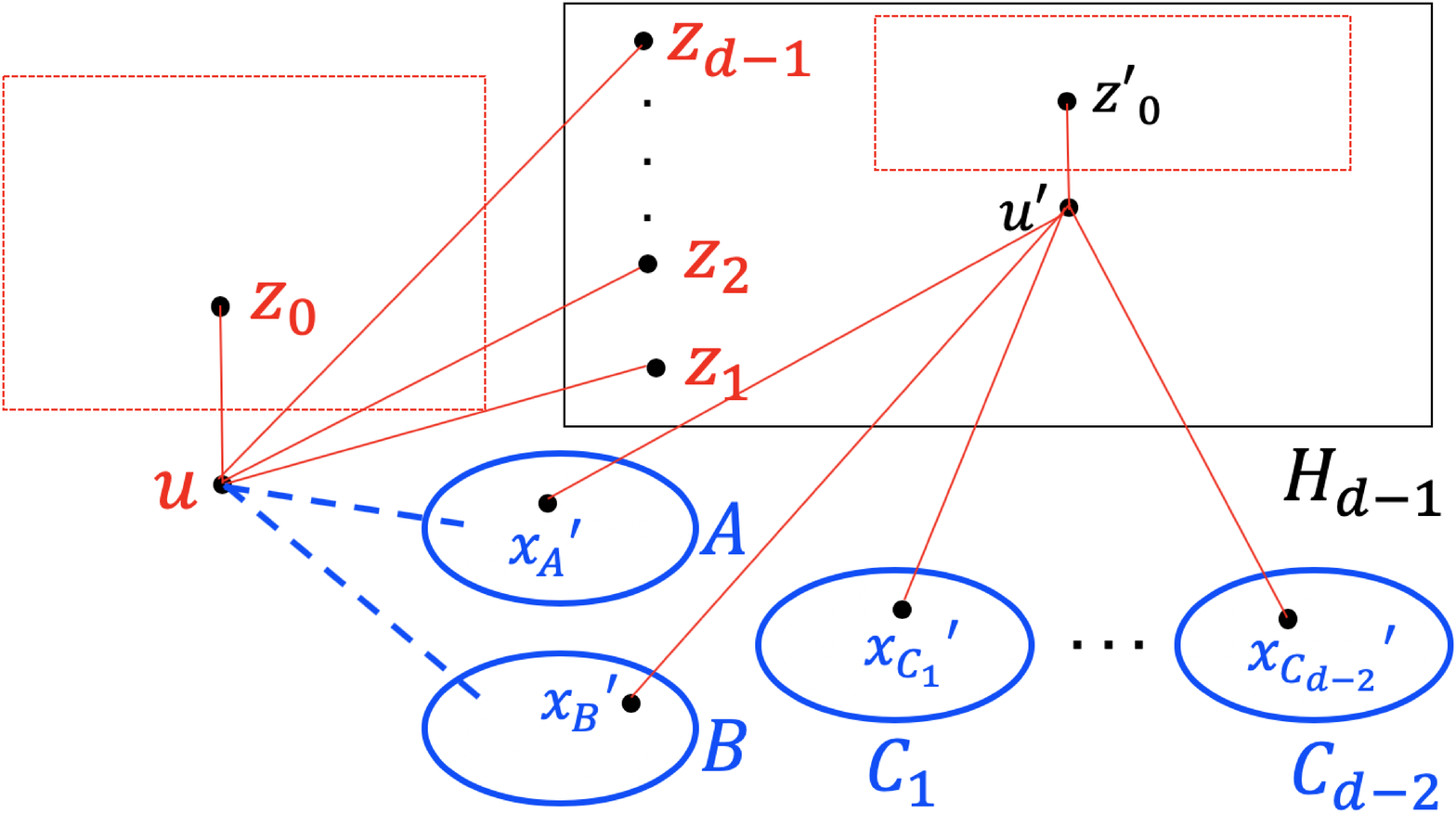}
    \caption{$H_{d-1}$ contains a red $T_n^{**}-\{z_1',z_2',\dots,z_d'\}$}
        \label{T10}
    \end{figure}
We now consider that $H_{d-1}$ contains a blue $K_{m-1}$, denoted by $C_{d-1}$ (see Figure \ref{T11}).
To avoid a blue copy of $2K_m$, there exists $x_{C_1}\in C_1$ such that $z_0x_{C_1}$ is red.
Otherwise, 
$\{z_0\}\cup C_1$ and $\{u\}\cup B$ form a blue copy of $2K_m$ and we are done.
Furthermore, there exist $x_A\in A,x_B\in B,x_{C_2}\in C_2,\dots,x_{C_{d-1}}\in C_{d-1}$, such that $x_{C_1}x_A,x_{C_1}x_B,x_{C_1}x_{C_2},\dots,x_{C_1}x_{C_{d-1}}$ are red.
Otherwise, 
$\{x_{C_1}\}\cup A$ and $\{u\}\cup B$ or $\{x_{C_1}\}\cup B$ and $\{u\}\cup A$ or $\{x_{C_1}\}\cup C_i$ and $\{u\}\cup A$ form a blue copy of $2K_m$, where $i\in[d-1]\setminus\{1\}$.
Then $T_n^{*}-\{b\}-\{u\}-\{z_1,\dots,z_{d-1}\}+\{x_{C_1}\}+\{x_A,x_B,x_{C_2},\dots,x_{C_{d-1}}\}$ contains a red copy of $T_n^{**}$ (see Figure \ref{T11}), and we are done.
\begin{figure}[H]
\centering
    \includegraphics[height=5cm]{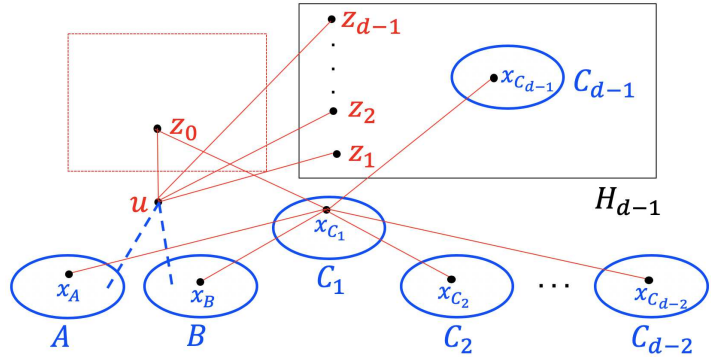}
    \caption{$T_n^{*}-\{b\}-\{u\}-\{z_1,\dots,z_{d-1}\}+\{x_{C_1},x_A,x_B,x_{C_2},\dots,x_{C_{d-1}}\}$ contains a red copy of $T_n^{**}$}
        \label{T11}
    \end{figure}
\end{proof}

Before giving the proof of Theorem \ref{Tn}, we show the following result given by  Sudarsana, Adiwijaya and Musdalifah \cite{SAM}.

\begin{lemma}\label{Pn}
Let $n\geq3$ and $m\geq2$ be positive integers, then $R(P_n,2K_m)=(n-1)(m-1)+2$.
\end{lemma}
\begin{proof}
We just need to prove the upper bound since the lower bound follows  by Theorem  \ref{Burr}  directly. 
Let $N=(n-1)(m-1)+2$ and the edges of $K_N$ be colored by red or blue.
We will show that $K_N$ contains a red $P_n$ or a blue $2K_m$.  
We first claim that there is a red $P_{n-1}$ contained in $K_N$. 
By the result given by Chv\'atal \cite{C}  that $R(T_n,K_m)=(n-1)(m-1)+1$, $K_N$ contains a red $P_n$ or a blue $K_m$.  
We just need to consider that $K_N$ contains a blue $K_m$ (denote it by $K$), otherwise we are done. 
Since $N-m=(n-2)(m-1)+1$ and we continue to apply the result of Chv\'atal \cite{C}  that $R(P_{n-1},K_m)=(n-2)(m-1)+1$,  then the graph induced on $V(K_N)-V(K)$ contains either a blue $K_m$ and we are done, or a red $P_{n-1}$. 
Hence, there is a red $P_{n-1}$ contained in $K_N$ and we denote it  by $P$. Let us label the end vertices of $P$ as $v_1$ and $v_2$.   
Now we use induction on $m$ to complete the proof. 
By a result of Chv\'atal and Harary \cite{CH},  the conclusion holds for $m=2$.  
Assume that the conclusion holds for $m-1$ and we will show that it holds for $m$. 
Let $H$ be the graph induced on $V(K_N)-V(P)$, then $v(H)=(n-1)(m-2)+2$. 
Since  the conclusion holds for $m-1$, then we obtain either a red $P_{n}$ and we are done, or a blue $2K_{m-1}$ in $H$. 
We denote the blue $2K_{m-1}$ by $A$. 
Now we consider the edges between $\{v_1,v_2\}$ and $V(A)$.  
To avoid a red $P_n$, all edges between $\{v_1,v_2\}$ and $V(A)$ are blue, which will imply a blue $2K_m$ and we are done. 
\end{proof}

{\bf Proof of Theorem \ref{Tn}} \ By  Lemma \ref{Pn}, we just need to consider that $n\geq4$ and $T_n$ is not a path.
The lower bound holds by Theorem \ref{Burr}.
Now we use induction on $m$ to prove the upper bound.
By 
a result of Chv\'atal and Harary \cite{CH},  the conclusion holds for the case that $m=2$.
Now we assume that the conclusion holds for $m-1$, i.e., $R(T,2K_{m-1})\leq (n-1)(m-2)+2$ holds for any tree $T$ with $n$ vertices.
Under this assumption, Lemma \ref{1} and Lemma \ref{2} guarantee that the property  “$2K_m$-good” is Stretching-preserving  and Expanding-preserving.
Applying Lemma \ref{Pn} that $P_n$ is $2K_m$-good and Corollary \ref{111}, we obtain that $R(T_n,2K_{m})\leq (n-1)(m-1)+2$.
\qed

\subsection{Obtain $T_n$ from $P_n$ by Stretching and Expanding} 
In this subsection, we give the proof of Proposition \ref{remark}. 
\

{\bf Proof of Proposition \ref{remark}} \ 
For any fixed tree $T_n$ on $n$ vertices, let $P$ be a longest path in $T_n$.
Let us label the end vertices of $P$ as $v$ and $w$, and label the vertices with degree greater than $2$ in the path as $u_1,u_2,\dots,u_l$.
Assume that the distance between $v$ and $u_1$, denoted by $d(v,u_1)$, is $t_0+1$, and $d(u_i,u_{i+1})=t_i+1$ for each $i\in[l-1]$ and $d(u_l,w)=t_l+1$. 
Let $d_i=d(u_i)-2$. 
We label the vertices not in $P$ and connecting to $u_i$ as $y_1^i,y_2^i,\dots,y_{d_i}^i$.
For each $i\in[l]$, let $B^i$ be a subset of $V(T_n)$ such that the subgraph induced on $\{u_i\}\cup B^i$ is a subtree of $T_n$ and this subtree just contain one vertex of $P$, i.e.,  $u_i$.  
Clearly, $V(T_n)=V(P)\cup B^1\cup\dots\cup B^l$ as in Figure \ref{T12}.
\begin{figure}[H]
\centering
    \includegraphics[height=4.5cm]{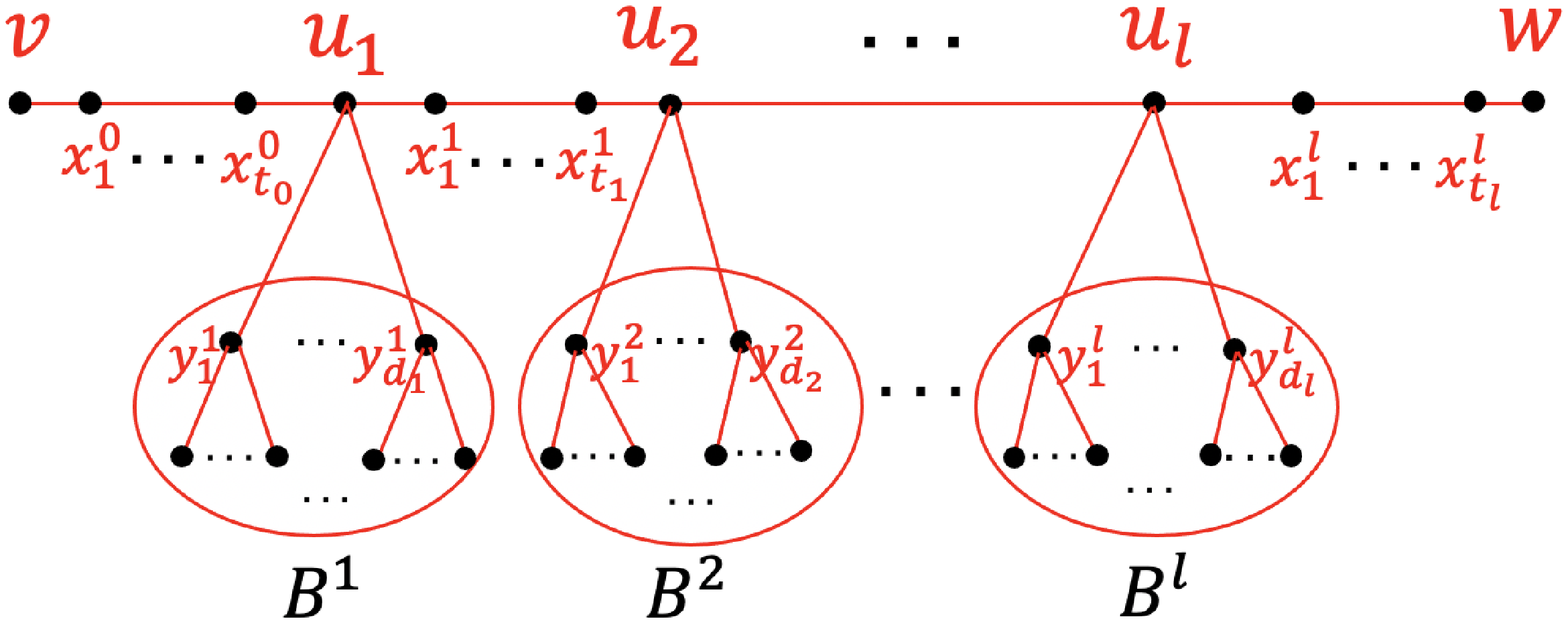}
     \caption{$T_n$}
     \label{T12}
    \end{figure}
    It's  easy to see that $P_n$ can be obtained from  $T_n$ by applying Stretching multiple times.   
     Actually, for any $i\in[l]$, the subgraph induced on $\{u_i\}\cup B^i$ is a subtree of $T_n$, thus there is at least one leaf in every such subtree. 
   Firstly, we can Stretching $T_n$ at the leaf $v$ by deleting a leaf $v'\in B^1$, then we lengthen path $P$ with   the end vertices as $v'$ and $w$. 
   We continue  Stretching this tree at the leaf $v'$ by deleting a leaf in $B^1-\{v'\}$ (if $B^1-\{v'\}=\emptyset$, then we deleting a leaf in $B^2$). 
   Continue this progress, we can   obtain  $P_n$ from  $T_n$ by applying Stretching multiple times (clearly the number of operations is equal to $|B^1|+\dots+|B^l|$).   
     
     Now we prove  that $T_n$ can be obtained from  $P_n$ by applying Stretching and Expanding multiple times.  We label the vertices of $P_n$ as in Figure \ref{T14} (if $t_1=0$, then $x_1^1$ is $u_2$). 
     We will delete all vertices in $\{b_{n-t_0-3},\dots,b_1\}$ in turn  and add all corresponding vertices to obtain $T_n$. 
     \begin{figure}[H]
\centering
    \includegraphics[height=1.7cm]{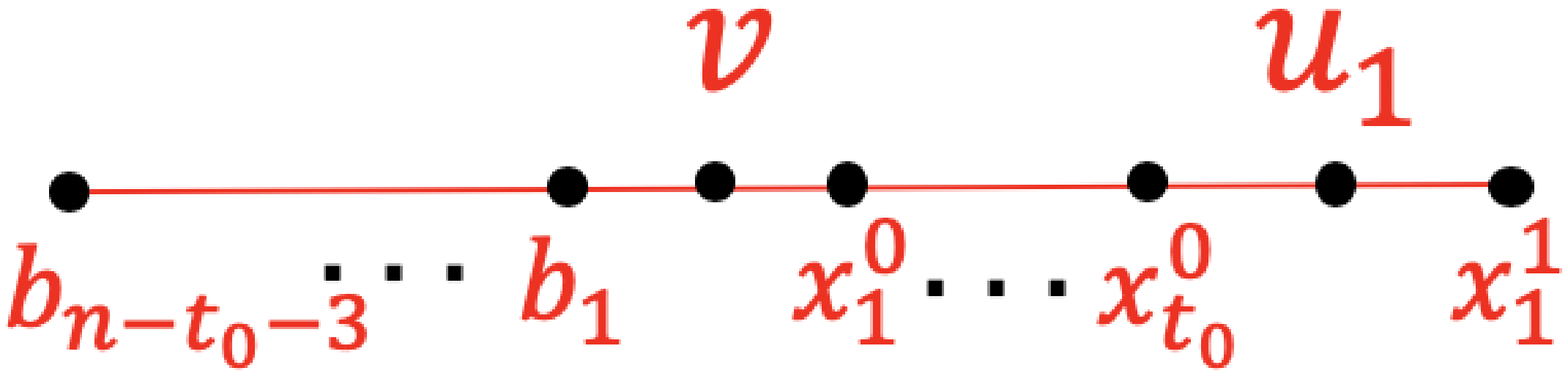}
    \caption{$P_n$}
        \label{T14}
    \end{figure}
     Our general operation is divided into the following steps. 
     
     Step $1$. By applying Stretching and Expanding multiple times, add all vertices in $B^1$.   
     We first relabel all vertices of $B^1$ in $T_n$. We relabel $y_{i}^1=i$ for all $i\in[d_1]$ and relabel all vertices in $B^1-\{1,2\dots,d_1\}$ as $d_1+1,d_1+2,\dots,|B^1|$ so that every $j$ is connected to exactly one of the vertices $1,2,\dots,j-1$, where $j\in[|B^1|]\setminus[d_1]$.  
       We define $$N_k(u_1)=\{j\in B^1:d(j,u_1)=k\}.$$ 
Denote $|N_k(u_1)|=N_k$ for all $k\geq 1$. 
   Clearly, $N_1(u_1)=\{1,2,\dots,d_1\}$, $N_2(u_1)=\{d_1+1,\dots,d_1+N_2\},\dots$.  
   We first show that how to delete $d_1$ vertices in $P_n$ and add all corresponding  vertices in $N_1(u_1)$.    
  We   apply Expanding $P_n$ at vertex $u_{1}$ $d_1$ times by deleting $b_{n-t_0-3},b_{n-t_0-4},\dots$ and $b_{n-t_0-(d_1+2)}$, and add $1,2,\dots$ and $d_1$ connecting to $u_{1}$ in turn, then we obtain $T'_n$ as in  Figure \ref{T15}.
\begin{figure}[H]
\centering
    \includegraphics[height=3.5cm]{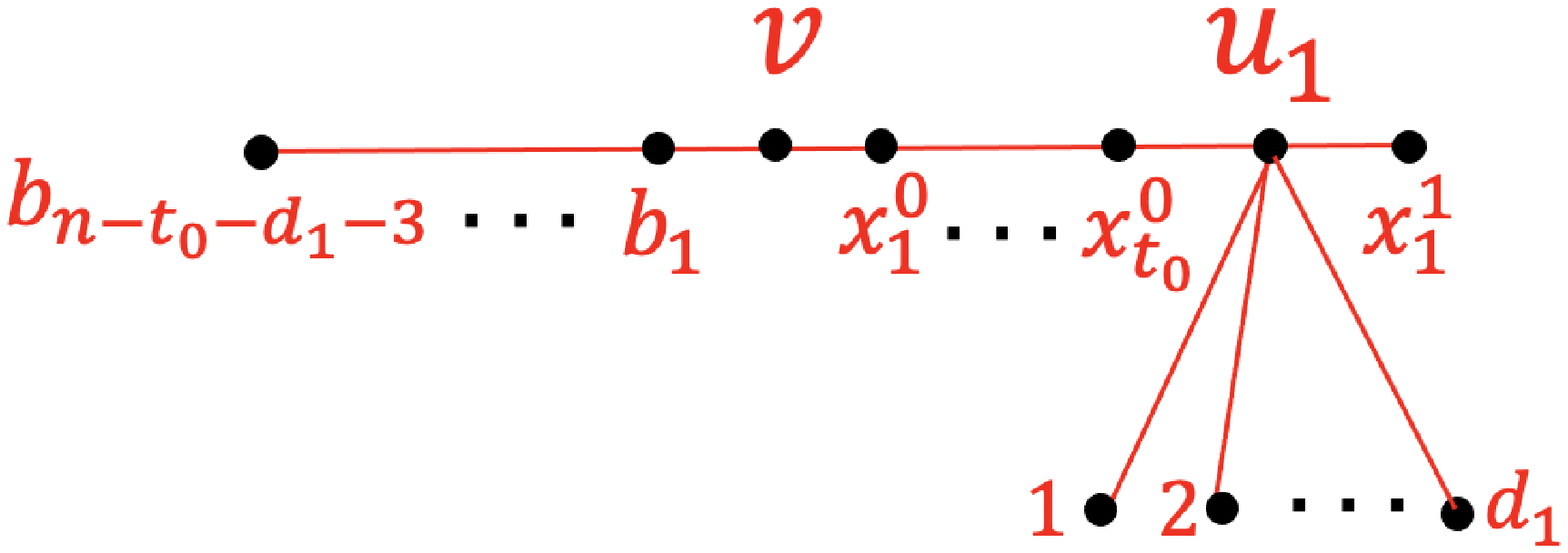}
    \caption{$T'_n$}
        \label{T15}
    \end{figure}
    \noindent       If for each $i\in[d_1]$, $i$ is a leaf of $T_n$, then we are done and go to step $2$. 
   Now we assume that there is $i$  not being  a leaf of $T_n$  for some $i\in[d_1]$. 
We assume that we have added all vertices in $N_k(u_1)$, and show that how to add all vertices in $N_{k+1}(u_1)$, where $k\geq1$.   
For $v\in N_k(u_1)$ so that the degree of $v$ in $T_n$ is more than $1$, 
we perform Stretching the tree at every such $v$ (in turn) by deleting an end point $b_\ell\in P_n$, where $\ell\in[n-t_0-d_1-3]$. 
On this foundation, for $v\in N_k(u_1)$ so that the degree of $v$ in $T_n$ is more than $2$, 
we perform Expanding the tree at every such $v$ $d_{T_n}(v)-2$ times (in turn) by deleting an end point $b_\ell\in P_n$ in turn, where $\ell\in[n-t_0-d_1-3]$. 
Hence we add all vertices in $N_{k+1}(u_1)$, where $k\geq1$.   

  Step $2$. By applying Stretching multiple times, add all vertices of $\{x^1_2,\dots,x^1_{t_1},u_2,x^2_1\}$.   
  We   first apply Stretching the tree obtained  by step $1$ at the end vertex $x^1_i\in  P_n$ (in turn) by deleting $b_{n-t_0-|B^1|-i-2}$, and add $x^1_{i+1}$ connecting to $x^1_i$ in turn, where $i\in[t_1-1]$.  
  Next we apply Stretching this tree at the end vertex $x^1_{t_1}$ by deleting $b_{n-t_0-|B^1|-t_1-2}$, and add $u_2$ connecting to $x^1_{t_1}$. 
  At last we apply Stretching this tree at the end vertex $u_2$ by deleting $b_{n-t_0-|B^1|-t_1-3}$, and add $x^2_1$ connecting to $u_2$. Hence, we obtain $T''_n$ as in  Figure \ref{T16}.
\begin{figure}[H]
\centering
    \includegraphics[height=4cm]{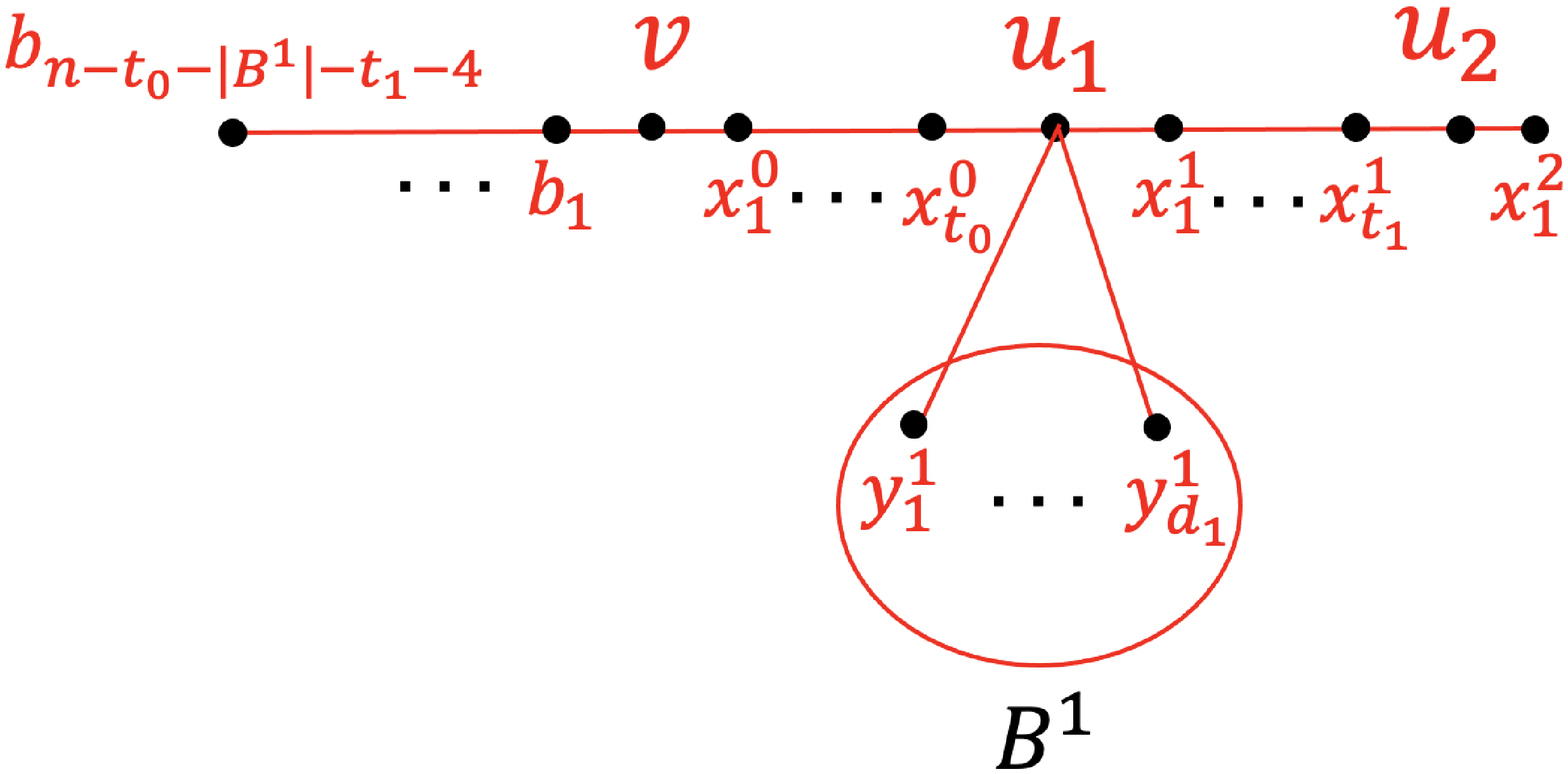}
    \caption{$T''_n$}
        \label{T16}
    \end{figure}
On this  foundation, we continue to attach $B_i$ to $u_i$ by performing Stretchings and Expandings similar to step $1$ and $2$ for each $i\in\{2,\dots,2l-1\}$  until we obtain $T_n$.

\qed

\section{Tree is $K_m\cup K_l$-good}
Firstly, we prove the following theorem which implies that $T_n$ is $K_m\cup K_l$-good for $n\geq 3$ and $m>l\geq2$.

\begin{theo}\label{Tnn}
Let $G$ be a $2K_{m-1}$-good  graph with $n\geq3$ vertices. 
If there exist two vertices $u$ and $v$ with degree one in $G$ such that  $G-\{u,v\}$ is $K_m$-good, 
then   $G$ is $K_m\cup K_l$-good, where $m>l\geq2$.
\end{theo}

\begin{proof}
By Theorem \ref{Burr}, we just need to prove the upper bound  that $R(G,K_m\cup K_l)\leq (n-1)(m-1)+1$.  
Let $N=(n-1)(m-1)+1$ and the edges of $K_N$ be colored by red or blue. 
We will show that $K_N$ contains a red $G$ or a blue $K_m\cup K_{l}$. 
Since  $R(G,2K_{m-1})=(n-1)(m-2)+2\leq N$, hence $K_N$ contains a red $G$ or a blue $2K_{m-1}$.  
We just need to consider that $K_N$ contains a  blue $2K_{m-1}$. Denote the vertex sets of these two disjoint copies of blue $K_{m-1}$ by $A$ and $B$. 
Let $F=K_N-A-B$, then  $|V(F)|=(n-3)(m-1)+1$. 
Since $R(G-\{u,v\},K_m)=(n-3)(m-1)+1$, then we just need to consider that $F$ contains a red $G-\{u,v\}$. Let $u'$ and $v'$ be  the vertices connecting with $u$ and $v$ in $G$, respectively (It's possible that $v'=u'$). 
 If the edges between $u'$ and $A$ or $v'$ and $B$ are all blue, then we have a blue  $K_m\cup K_{m-1}$. 
 Otherwise,  there is at least  one red edge  between $u'$ and $A$, and at least one red edge between $v'$ and $B$, so there is a red $G$ and we are done. 
\end{proof}

Taking $G$ to be $T_n$ in Theorem \ref{Tnn}, and applying Theorems \ref{Tn} and \ref{Tnn}, we obtain the following result.
\begin{coro}\label{coro}
$T_n$ is $K_m\cup K_l$-good for $n\geq3$ and $m>l\geq2$.
\end{coro}

\section{The Ramsey number of a forest versus a disjoint union of complete graphs}
In \cite{St}, Stahl  applied Chv\'atal's theorem about the exact value of $R(T_n,K_m)$ to determine the Ramsey number of a forest versus $K_m$. 
In this section, we also obtain the Ramsey number $R(F,K_m\cup K_l)$ by applying  Theorems \ref{Tn} and Corollary \ref{coro}, where $F$ is a forest and $m,l\geq2$. 
Indeed, we prove a general result under some conditions. 
Before that, we extend the Ramsey goodness of  connected graphs to all graphs. 
Close behind the lower bound for a connected graph $G$ and $H$ with $|V(G)|\geq s(H)$ given by Burr \cite{B}  that $R(G,H)\geq (|V(G)|-1)(\chi(H)-1)+s(G)$, 
Gould and Jacobson \cite{GJ} observed a construction which yields a general lower bound of $R(\mathcal{F},H)$ for arbitrary graph $\mathcal{F}$ and $H$. 

\begin{theo}{\rm (Gould-Jacobson \cite{GJ})}\label{GJ}
Let $\mathcal{F}$ be a graph, let $F_1,F_2,\cdots, F_{r-1}$ and $F_r$ be the  components (maximal  connected subgraphs) of $\mathcal{F}$. 
Let  $I$ be the set of orders of $F_1,F_2,\cdots, F_r$, i.e., for every $i\in I$, there exists at least a graph $F_l$ with $|V(F_l)|=i$, $l\in[r]$. 
Let $n(\mathcal{F})$ be the maximum element in $I$ and $k_i(\mathcal{F})$ be the number of components with order $i$.  
Let $H$ be a graph, let  $\beta_i=R(F_i,H)-(|V(F_i)-1|)(\chi(H)-1)-s(H)$ and 
$$p=\max_{ j\in I}\Bigg\{(j-1)(\chi(H)-2)+\sum_{i=j}^{n(\mathcal{F})}ik_i(\mathcal{F})\Bigg\}+s(H)-1.$$  then $$p\leq R(\mathcal{F},H)\leq p+\max_{i}(\beta_i).$$ 
\end{theo}

We say that a graph $\mathcal{F}$ is $H$-good if the equality  $$R(\mathcal{F},H)= \max_{ j\in I}\Bigg\{(j-1)(\chi(H)-2)+\sum_{i=j}^{n(\mathcal{F})}ik_i(\mathcal{F})\Bigg\}+s(H)-1$$ holds under the conditions of Theorem \ref{GJ} and $s(H)\geq j_0$, where $j_0\in I$ be the maximum integer satisfying $(j_0-1)(\chi(H)-2)+\sum_{i=j_0}^{n(\mathcal{F})}ik_i(\mathcal{F})=\max_{ j\in I}\left\{(j-1)(\chi(H)-2)+\sum_{i=j}^{n(\mathcal{F})}ik_i(\mathcal{F})\right\}$. 
This definition is consistent for a connected graph with the definition of Ramsey-goodness given by Burr. 
We extend the upper bound given by Theorem \ref{GJ}. 
Before giving the upper bound  of $R(\mathcal{F},H)$ (Theorem \ref{mathcal{F}}), 
we remark the following consequence which is implicit in  Theorem \ref{GJ}. 
\begin{remark}\label{5.1}
Let  $H$ be a graph and let $F$ be the disjoint union of graphs $F_1, F_2,\dots, F_k$,  where each of  $F_1,F_2,\cdots, F_k$ has $n$ vertices. 
Then 
$$R(F,H)\leq \max_{i\in[k]}\left\{R(F_i,H)\right\}+n(k-1).$$
 Moreover,  if  $F_i$ is a connected and $H$-good graph with $n\geq s(H)$ for each $i\in[r]$, then  $F$ is $H$-good.  
\end{remark}

\begin{theo}\label{mathcal{F}}
Let $H$ be a graph, let $\mathcal{F}$ be the disjoint union of  graphs $F_1,F_2,\cdots, F_{r-1}$ and $F_r$. 
Let  $I$ be the set of orders of $F_1,F_2,\cdots, F_{r}$, i.e., for every $i\in I$, there exists at least a graph $F_l$ with $|V(F_l)|=i$, $l\in[r]$. 
Let $n(\mathcal{F})$ be the maximum element in $I$ and $k_i(\mathcal{F})$ be the number of graphs with order $i$. Then
$$R(\mathcal{F},H)\leq\max_{ j\in I}\Bigg\{\max_{|V(F_p)|=j,p\in[r]}\left\{R(F_p,H)\right\}+\sum_{i=j}^{n(\mathcal{F})}ik_i(\mathcal{F})-j\Bigg\}.$$ 
Moreover, if  $F_i$ is a connected and $H$-good graph with $|V(F_i)|\geq s(H)$ for each $i\in[r]$, then  $\mathcal{F}$ is $H$-good.  
\end{theo}

\begin{proof}
 Let  $$N=\max_{ j\in I}\Bigg\{\max_{|V(F_p)|=j,p\in[r]}\left\{R(F_p,H)\right\}+\sum_{i=j}^{n(\mathcal{F})}ik_i(\mathcal{F})-j\Bigg\}$$ and the edges of $K_N$ be colored by red or blue. 
Now we  prove the upper bound by assuming that there is no blue $H$ and using the descending  induction to show the existence of a red $\mathcal{F}$ in $K_N$.   
Let $G_j$ be the  graph consisting of all $F_i$ with order at least $j$.  
Clearly, 
\begin{eqnarray*}
N&\geq& \max_{|V(F_i)|=n(\mathcal{F})}\left\{R(F_i,H)\right\}+n(\mathcal{F})k_{n(\mathcal{F})}(\mathcal{F})-n(\mathcal{F})\\
&=& \max_{|V(F_i)|=n(\mathcal{F})}\left\{R(F_i,H)\right\}+n(\mathcal{F})(k_{n(\mathcal{F})}(\mathcal{F})-1)\\
& \overset{Remark\; \ref{5.1}}{\geq}& R(G_{n(\mathcal{F})},H).
\end{eqnarray*}
Apply the assumption that there is no blue $H$, there is a red $G_{n(\mathcal{F})}$ in $K_N$. 
So the base case holds.  
Now  assume that there is a red $G_{j+1}$ in $K_N$, our goal is to show that there is a red $G_{j}$ in $K_N$.  
Since $|V(G_{j+1})|=\sum_{i=j+1}^{n(\mathcal{F})}ik_i(\mathcal{F})$, then the order of $K_N-G_{j+1}$ is
\begin{eqnarray*}
N-\sum_{i=j+1}^{n(\mathcal{F})}ik_i(\mathcal{F})&\geq& \max_{|V(F_p)|=j,p\in[r]}\left\{R(F_p,H)\right\}+\sum_{i=j}^{n(\mathcal{F})}ik_i(\mathcal{F})-j-\sum_{i=j+1}^{n(\mathcal{F})}ik_i(\mathcal{F})\\
&=& \max_{|V(F_p)|=j,p\in[r]}\left\{R(F_p,H)\right\}+jk_j(\mathcal{F})-j\\
&=& \max_{|V(F_p)|=j,p\in[r]}\left\{R(F_p,H)\right\}+j(k_j(\mathcal{F})-1)\\
& \overset{Remark\; \ref{5.1}}{\geq}& R(G_j-G_{j+1},H).
\end{eqnarray*}
Hence there is a red $G_j-G_{j+1}$ in $K_N-G_{j+1}$, which yields a red $G_j$ in $K_N$. 
By induction  the proof of the first part of the theorem  is completed. 

If   $F_i$ is a connected and $H$-good graph with $|V(F_i)|\geq s(H)$ for each $i\in[r]$, 
then $R(\mathcal{F},H)\leq\max_{ j\in I}\left\{(j-1)(\chi(H)-1)+s(H)+\sum_{i=j}^{n(\mathcal{F})}ik_i(\mathcal{F})-j\right\}.$ 
Combining with the lower bound of Theorem \ref{GJ}, we obtain that $\mathcal{F}$ is $H$-good.  
\end{proof}

By Corollary \ref{coro} and Theorem \ref{mathcal{F}}, we can  obtain the following exact value of $R(F,K_m\cup K_l)$, where $F$ is a forest and $m,l\geq2$, and a forest is $K_m\cup K_l$-good. 

\begin{coro}
Let $m\geq l \geq 2$ be positive integers. 
Let $F$ be a forest, then $F$ is $K_m\cup K_l$-good.  
\end{coro}

\section{Remarks}
Determining $R(G,H)$ in general is a very challenging problem.  
When we focus on the  problem related to Ramsey goodness, we are interested in exploring what kind of conditions can yield good Ramsey goodness property of graphs.  
As mentioned in Section $1$, nice work in \cite{BPS, C, CH, PS} concern this type of conditions. 
Theorem \ref{mathcal{F}} shows that if every component of a disconnected graph $F$ is $H$-good, then $F$ is $H$-good.  
It is natural to study whether $F$ is $H$-good and  $G$-good implies that $F$ is $H\cup G$-good. 
Clearly, this is not true for graphs such as $K_n$. 
An obvious result is that any connected graph $F$ is $K_2$-good, 
but $R(K_n,2K_2)=n+2$, which implies that $K_n$ is not $2K_2$-good.  
Certainly, some results  provide evidence for that the answer of this question is yes for some  $F$.   
Sudarsana \cite{Su} showed that $P_n$ with $n\geq (t-2)((tm-2)(m-1)+1)+3$ is $tK_m$-good for  $m$ and $t\geq2$ be integers.  
Indeed, the condition on the number of vertices $n$ can not  be released completely.  
For example,  we can prove that $T_n$ is not $tK_2$-good when $n\leq t$ (actually we determined the exact Ramsey number $R(T_n,tK_2)$ in \cite{HP}).   
It's interesting to explore what kind of graph $F$ can satisfy that $F$ is $H$-good and $G$-good implies that $F$ is $H\cup G$-good when the number of vertices in $F$ is sufficiently large.

\end{document}